\documentclass[12pt]{amsart}
\usepackage{graphicx}
\usepackage{amsmath,amsthm, amsfonts,amssymb, stmaryrd,yfonts,pxfonts,pifont,eufrak, bbm}
\newtheorem{Theorem}{Theorem}

\newtheorem{Cor}{Corollary}
 \newtheorem{Lemma}{Lemma}

 \newtheorem{Proposition}{Proposition}
 \theoremstyle{definition}
 
 \theoremstyle{remark}
 \newtheorem{Remark}[Lemma]{Remark}
 \numberwithin{equation}{subsection}

\begin{document}
\title[MULTIDIMENSIONAL SHIFTS AND FINITE MATRICES]{MULTIDIMENSIONAL SHIFTS AND FINITE MATRICES}%
\author{Puneet Sharma AND Dileep Kumar}
\address{Department of Mathematics, I.I.T. Jodhpur, Old Residency Road, Ratanada, Jodhpur-342011, INDIA}%
\email{puneet@iitj.ac.in, pg201383502@iitj.ac.in}%




\begin{abstract}
Let $X$ be a $2$-dimensional subshift of finite type generated by a finite set of forbidden blocks (of finite size). We give an algorithm for generating the elements of the shift space using sequence of finite matrices (of increasing size). We prove that the sequence generated yields precisely the elements of the shift space $X$ and hence characterizes the elements of the shift space $X$. We extend our investigations to a general $d$-dimensional shift of finite type.  In the process, we prove that that elements of $d$-dimensional shift of finite type can be characterized by a sequence of finite matrices (of increasing size).
\end{abstract}
\maketitle

\section{INTRODUCTION}

Dynamical systems have been long used to determine the long term behaviour of various natural processes occurring around us. Many systems around us have been successfully modelled using the theory of dynamical systems and long term behavior for such systems has been approximated with desired accuracy. While logistic system has been used to model population dynamics of many species, systems like lorenz model have been successfully used to model weather prediction. The theory has found applications in areas such as climate control, population dynamics and computational neuroscience \cite{lor,iz,hsig}. Consequently, theoretical studies in the field of dynamical systems have gained attention and many interesting results have been obtained. Such studies have not only enriched the literature but has also provided greater insight into the theory of dynamical systems. In recent times, symbolic dynamical systems have been used to investigate general theory of dynamical systems. The simpler structure and easier visualization makes it an effective tool to investigate the dynamics of the underlying system. While Jacques Hadamard applied the theory of symbolic dynamics to study the geodesic flows on surfaces of negative curvature \cite{had}, Claude Shennon used symbolic dynamics to develop the mathematical theory of communication systems \cite{shanon}. In recent times, the topic has found applications in areas like data storage, data transmission, control networks and modelling of gene networks\cite{lind1,song,stro}. In recent times, multidimensional symbolic dynamics has been topic of interest and has attracted attention of several researchers across the globe. As many of the systems cannot be modelled using one dimensional symbolic systems, multidimensional symbolic systems can been used to model dynamically complex systems and investigate dynamical complexities of a system that cannot be investigated using one dimensional systems. The topic has found applications in areas like cohomology theory and tiling of finite dimensional planes. \\

In one of the early works, Berger investigated multidimensional subshifts of finite type over finite number of symbols. He proved that for a multidimensional subshift, it is algorithmically undecidable whether an allowed partial configuration can be extended to a point in the multidimensional shift space \cite{ber}. In \cite{rob}, the author  provides an example to prove that a multidimensional shift space may not contain any periodic points. The investigations motivated further research in this area and many interesting results have been obtained \cite{quas,ban,sc,beal,boy1,hoch1,hoch4,sam}. In \cite{quas}, authors proved that multidimensional shifts of finite type having positive topological entropy cannot be minimal. In particular, they proved that any subshift of finite type with positive topological entropy contains a subshift which is not of finite type and hence contains infinitely many subshifts of finite type \cite{quas}. In \cite{ban}, the authors investigated mixing properties of multidimensional shift of finite type. In \cite{boy1}, authors investigate mixing $\mathbb{Z}^d$ shifts of finite type and sofic shifts with large entropy. They give examples to prove that while there exists $\mathbb{Z}^d$ mixing systems such that no non-trivial full shift is a factor for such systems, there exists sofic systems such that the only minimal subsystem is a singleton. In \cite{hoch1}, authors prove that a real number the set of entropies for $\mathbb{Z}^d$ shifts of finite type ($d\geq 2$) coincides with the set of the infimum of all possible recursive sequences of rational numbers. In \cite{hoch4}, the author improved the result and proved that $h\geq 0$ is the entropy of a $\mathbb{Z}^d$ effective dynamical system if and only if it is the lim inf of a recursive sequence of rational numbers. The problem of determining the class of shifts having a dense set of periodic points is still open. In \cite{sam}, the author proves that for a two dimensional shift sopace, strongly irreducible shifts of finite type have dense set of periodic points. However, the problem is still open for multidimensional shift spaces of dimension greater than two.  \\

Let $A = \{a_i : i \in I\}$ be a finite set and let $d$ be a positive integer. Let the set $A$ be equipped with the discrete metric and let $A^{\mathbb{Z}^d}$, the collection of all functions $c : \mathbb{Z}^d \rightarrow A$ be equipped with the product topology. Any such function $c$ is called a configuration over $A$. The function $\mathcal{D} : A^{\mathbb{Z}^d} \times A^{\mathbb{Z}^d} \rightarrow \mathbb{R}^+$ be defined as $\mathcal{D} (x,y) = \frac{1}{n+1}$, where $n$ is the least non-negative integer such that $x \neq y$ in $R_n = [-n,n]^d$, is a metric on $A^{\mathbb{Z}^d}$ and generates the product topology. For any $a\in \mathbb{Z}^d$, the map $\sigma_a : A^{\mathbb{Z}^d} \rightarrow A^{\mathbb{Z}^d}$ defined as $(\sigma_a (x))(k)= x(k+a)$ is a $d$-dimensional shift and is a homeomorphism. For any $a,b\in \mathbb{Z}^d$, $\sigma_a \circ \sigma_b = \sigma_b \circ \sigma_a$ and hence $\mathbb{Z}^d$ acts on $A^{\mathbb{Z}^d}$ through commuting homeomorphisms. A set $X \subseteq A^{\mathbb{Z}^d}$ is $\sigma_a$-invariant if $\sigma_a(X)\subseteq X$. Any set $X \subseteq A^{\mathbb{Z}^d}$ is shift-invariant if it is invariant under $\sigma_a$ for all $a \in {\mathbb{Z}}^d$. A non-empty, closed shift invariant subset of $A^{\mathbb{Z}^d}$ is called a shift space. If $Y \subseteq X$ is a nonempty, closed shift invariant subset of $X$, then $Y$ is called a subshift of $X$. A closed shift invariant subset $X\subset A^{\mathbb{Z}^d}$ is a shift of finite type if and only if there exists a finite set $\mathbb{F}\subset \mathbb{Z}^d$ and a set $\mathbb{P}\subset A^{\mathbb{F}}$ such that $X=\{x \in A^{\mathbb{Z}^d} : (\pi_{\mathbb{F}}\circ \sigma_n)(x)\in \mathbb{P} \text{~~for any~~} n\in \mathbb{Z}^d \}$, where $\pi_{\mathbb{F}}(x)$ is the restriction of $x\in A^{\mathbb{Z}^d}$ to $\mathbb{F}$. For any nonempty $S\subset \mathbb{Z}^d$, an element in $A^S$ is called a pattern over $S$. A pattern is said to be finite if it is defined over a finite subset of $\mathbb{Z}^d$. A pattern $q$ over $S$ is said to be an extension of the pattern $p$ over $T$ if $T\subset S$ and $q|_T=p$. The extension $q$ is said to be proper extension if $S\cap Bd(T)=\phi$, where $Bd(T)$ denotes the boundary of $T$. Let $\mathcal{F}$ be a given set of finite patterns (possibly over different subsets of $\mathbb{Z}^d$) and let $X=\overline{\{x\in A^{\mathbb{Z}^d}: \text{any pattern from~~} \mathcal{F} \text{~~does not appear in~~} x \}}$. The set $X$ equivalently defines a subshift of $\mathbb{Z}^d$ generated by set of forbidden patterns $\mathcal{F}$. If the set $\mathcal{F}$ is a finite set of finite patterns, we say that the shift space $X$ is a shift of finite type. We say that a pattern is allowed if it is not an extension of any forbidden pattern. We denote the shift space generated by the set of forbidden patterns $\mathcal{F}$ by $X_{\mathcal{F}}$. Two forbidden sets $\mathcal{F}_1$ and $\mathcal{F}_2$ are said to be equivalent if they generate the same shift space, i.e. $X_{\mathcal{F}_1}= X_{\mathcal{F}_2}$. A forbidden set $\mathcal{F}$ of patterns is called minimal for the shift space $X$ if $\mathcal{F}$ is the set with least cardinality such that $X=X_{\mathcal{F}}$. It is worth mentioning that a shift space $X$ is of finite type if and only if its minimal forbidden set is a finite set of finite patterns.\\




For any two distinct $k$-tuples $(a_1,a_2,\ldots,a_k), (b_1,b_2,\ldots,b_k)$ over $\mathbb{R}$, we say that $(a_1,a_2,\ldots,a_k)<(b_1,b_2,\ldots,b_k)$ if $a_r<b_r$ where $r=\min\{i : a_i\neq b_i\}$. The relation defines a total order on $\mathbb{R}^k$ and is known as the dictionary order on $\mathbb{R}^k$. Let $O_{\mathcal{H}}$ (and $O_{\mathcal{V}}$) be the restriction of the dictionary order on set of $2\times 2$ matrices, when any matrix $\left(\begin{array}{cc} a & b \\ c & d \end{array}\right)$ is represented as $(a,b,c,d)$ (and $(a,c,b,d)$ respectively). Analogously, let $O_{\mathcal{H}}$ and $O_{\mathcal{V}}$ be the orders defined on the set of $r\times s$ matrices obtained by restricting dictionary order, when any matrix is represented as a tuple by reading entries row-wise (left to right) and column-wise (top to bottom) respectively.\\

Let $M$ be a $k^2\times k^2$ matrix of the form \\

\centerline{{$M = \left(%
\begin{array}{cccc}
  M_{11} & M_{12} & \ldots & M_{1k} \\
  M_{21} & M_{22} & \ldots & M_{2k} \\
  \vdots & \vdots &        & \vdots  \\
  M_{k1} & M_{k2} & \ldots & M_{kk} \\
\end{array}%
\right)$}}
\vskip 0.5cm

where each $M_{ij}$ is a $k\times k$ matrix. For a pair of matrices $(M_{ij},M)$, assign a $1\times k^4$ matrix $(M_{ij}\bigotimes M)$ as \\

\centerline{$(M_{ij}\bigotimes M)_{1r}= (M_{ij})_{\alpha\beta} (M_{\alpha\beta})_{\gamma\delta}$}
\vskip 0.25cm
where $\alpha,\beta,\gamma,\delta\in\{1,2,\ldots,k\}$ and $(\alpha,\beta,\gamma,\delta)$ is the unique solution to the equation $r= \alpha k^3+\beta k^2+\gamma k +\delta -(k^3+k^2+k)$. \\

It may be noted that the if the $1\times k^4$ row vector is visualised as $k^2$ groups of $k^2$ elements, then the first $k^2$ entries of the resultant $(M_{ij}\bigotimes M)$ are obtained by multiplying $(M_{ij})_{11}$ with $k^2$ entries of $M_{11}$ (entries read rowwise), next $k^2$ entries are obtained by multiplying $(M_{ij})_{12}$ with $k^2$ entries of $M_{12}$ and so on. In general, for $n\in\{1,2,\ldots,k^2\}$, the $n$-th group is determined by multiplying $(M_{ij})_{wz}$ by $k^2$ entries of $M_{wz}$ where $(w,z)\in\{1,2,\ldots,k\}^2$ is unique solution to the equation $n=(w-1)k+z$. Consequently, the operation $\bigotimes$ is well defined and assigns a $k^4$ row vector for any input pair $(M_{ij},M)$. It may be noted that although the operation is defined for the pair $(M_{ij},M)$ where $M_{ij}$ is submatrix of $M$, the operation is well defined for any pair $(P,Q)$ where $P$ and $Q$ are square matrices of order $k$ and $k^2$ respectively.\\

In this paper, we address the problem of characterizing the elements of a two-dimensional shift of finite type generated by a given set of forbidden patterns. In particular, given a two-dimensional shift space characterized by a finite set of forbidden blocks, we provide an algorithm for characterizing the elements of the shift space $X$ using a sequence of finite matrices (of increasing size). The algorithm generates all possible elements of the shift space using the matrices generated and hence determines the shift space completely. We generalize our algorithm to a $d$-dimensional shift of finite type. In the process, we prove that that elements of $d$-dimensional shift of finite type can be characterized by a sequence of finite matrices of increasing size. \\

\section{Main Results}

\begin{Proposition}\label{sb}
$X$ is a $d$-dimensional shift of finite type $\implies$  there exists a set $\mathcal{C}$ of $d$-dimensional cubes such that $X=X_{\mathcal{C}}$.
\end{Proposition}

\begin{proof}
Let $X$ be a shift of finite type and let $\mathcal{F}$ be the minimal forbidden set of patterns generating the shift space $X$. Then, $\mathcal{F}$ contains finitely many patterns defined over finite subsets of $\mathbb{Z}^d$. For any $p\in\mathcal{F}$, let $l_p^i$ be the length of the pattern $p$ in the i-th direction. Let $l_p = \max\{l_P^i : i=1,2,\ldots,d\}$ denote the width of the pattern $p$ and let $l= \max\{ l_p : p \in \mathcal{F}\}$. Let $\mathbb{E}_{\mathcal{F}}$ denote the set of extensions of patterns in $\mathcal{F}$ and $\mathbb{C}_l$ be the collection of $d$-dimensional cubes of length $l$. Note that if $p$ is a pattern with width $m$, forbidding a pattern $p$ for $X$ is equivalent to forbidding all extensions $q$ of $p$ in $\mathcal{C}_m$. Consequently, each pattern in the forbidden set of width $l$ can be replaced by an equivalent forbidden set of cubes of length $l$ and $\mathcal{C}= \mathbb{C}_l \cap \mathbb{E}_{\mathcal{F}}$ is an equivalent forbidden set for the shift space $X$. Consequently, $X=X_{\mathcal{C}}$ and the proof is complete.
\end{proof}

\begin{Remark}
The above result establishes that every $d$-dimensional shift of finite type can equivalently be generated by a set of cubes of fixed finite length. Such a modification generates an equivalent forbidden set which in general is not minimal. The above result establishes the existence of an equivalent forbidden set by considering all the cubes which are extension of the set of patterns in $\mathcal{F}$. It may be noted that the cardinality of the new set can be further optimized by considering only those cubes which are not proper extensions of patterns in $\mathcal{F}$ (but are of same size $l$). Such a consideration reduces the cardinality of the forbidden set further and hence reduces the complexity of the original system. It is worth mentioning that the forbidden set obtained by stated method is still not minimal. However, the $d$-dimensional cubes generated by such a mechanism are of same size which can further be used for generating (characterizing) the elements of the shift space $X$. We say that a shift of finite type $X$ is generated by cubes of length $l$ if there exists a set of cubes $\mathcal{C}$ of length $l$ such that $X=X_{\mathcal{C}}$.
\end{Remark}

\begin{Theorem}\label{2d}
For every $2$-dimensional shift of finite type, there exists a sequence of finite matrices characterizing the elements of $X$.
\end{Theorem}

\begin{proof}
Let $X$ be the shift space generated by forbidding given set of patterns. In this proof, we provide an algorithm for generating finite matrices (of increasing size) characterizing the existence of arbitrary large squares (rectangles) allowed for the shift space $X$. The sequence of matrices in turn characterizes the elements of the shift space $X$ and thus addresses the non-emptiness problem for the shift space generated by a given set of forbidden blocks. We now describe the proposed algorith below: \\

{\bf Step 1: Computation of Generating Matrices \\}

Let $X$ be a $2$-dimensional shift of finite type and let $\mathcal{P}$ be the finite set of forbidden patterns characterizing the shift space $X$. By Lemma \ref{sb}, there exists a set of forbidden squares $\mathcal{S}$ (all of same size, say $l$) generating the shift space $X$.  Let $B_1,B_2,\ldots,B_k$ be the the collection of all squares of size $~~l$. Let $V_0$ be the $k\times k$ matrix (indexed by $B_1,B2,\ldots,B_k$)  indicating vertical compatibility of the squares $B_i$. For notational convenience, let the index $B_i$ be denoted by $i, ~~\forall i=1,2,\ldots,k$. Then,

\begin{center}
$V^0_{ij} = \left\{%
\begin{array}{ll}
    1, & \hbox{$\left(\begin{array}{c} B_i \\ B_j \end{array}\right) \text{~~is allowed in~~} X$} \\
    0  & \hbox{\text{~~otherwise~~}} \\
\end{array}%
\right.$
\end{center}

Let $H^0$ be a matrix indexed by the set of rectangles of size $2l\times l$ indicating horizontal compatibility of rectangles of size $2l\times l$. As any rectangle of size $2l\times l$ is of the form $\left(\begin{array}{c} B_i \\ B_j \end{array}\right)$, $H^0$ is a $k^2\times k^2$ matrix indexed by rectangles of the form  $\left(\begin{array}{c} B_i \\ B_j \end{array}\right)$. For notational convenience, let the index $\left(\begin{array}{c} B_i \\ B_j \end{array}\right)$ be denoted by $(ij), ~~\forall i,j\in \{1,2,\ldots,k\}$. Then,

\begin{center}
$H^0_{(ij)(rs)} = \left\{%
\begin{array}{ll}
    1, & \hbox{$\left(\begin{array}{cc} B_i & B_r \\ B_j & B_s \end{array}\right) \text{~~is allowed in~~} X$} \\
    0  & \hbox{\text{~~otherwise~~}} \\
\end{array}%
\right.$
\end{center}

It may be noted that while entries of the matrix $V_0$ indicate the vertical compatibility of squares of length $l$, entries of the matrix $H_0$ characterize the existence of squares of side $2l$. In the next step, we use the matrices computed to verify the validity of rectangles (squares) of size $4l\times 2l$ ($4l\times 4l$).\\

{\bf Step 2: One Step Extension: Computing $V_1$ and $H_1$ \\}

In the first step, we computed the matrices $V_0$ and $H_0$ characterizing existence of rectangles and squares of size $2l\times l$ and $2l\times 2l$ respectively. To add clarity to the structure of the matrix $H^0$, let the index set of $H^0$ follow the dictionary order, i.e., let the index set of $H_0$ be ordered as $\{(11),(12)\ldots,(1k),(21),(22),\ldots,(2k),\ldots,(k1),(k2), \ldots, (kk)\}$. Consequently, $H^0$ can be viewed as a block matrix of the form \\

\centerline{{$H^0 = \left(%
\begin{array}{cccc}
  R^0_{11} & R^0_{12} & \ldots & R^0_{1k} \\
  R^0_{21} & R^0_{22} & \ldots & R^0_{2k} \\
  \vdots & \vdots &        & \vdots  \\
  R^0_{k1} & R^0_{k2} & \ldots & R^0_{kk} \\
\end{array}%
\right)$}}
\vskip 0.5cm

where $R^0_{ij}$ is a $k\times k$ matrix whose entries characterize the squares of the form $\left(\begin{array}{cc} B_i & B_j \\ B_r & B_s \end{array}\right)$. More precisely, $(r,s)$-th entry of $R^0_{ij}$ is $1$ if and only if $\left(\begin{array}{cc} B_i & B_j \\ B_r & B_s \end{array}\right)$ is allowed. Equivalently, the entries of the matrices $R^0_{ij}$ characterize all squares of size $2l$ whose top half (which is rectangle of size $l\times 2l$) is $(B_i~~ B_j)$.\\

Let $V^1$ be $k^4\times k^4$ matrix indexed by squares of size $2l$ indicating vertical compatibility (of the squares of size $2l$). As any square of size $2l$ is of the form $\left(\begin{array}{cc} B_i & B_j \\ B_r & B_s \end{array}\right)$, $V^1$ is equivalently indexed by the squares of the form $\left(\begin{array}{cc} B_i & B_j \\ B_r & B_s \end{array}\right)$. For notational convenience, let the index $\left(\begin{array}{cc} B_i & B_j \\ B_r & B_s \end{array}\right)$ be denoted by $(ijrs)~~\forall 1\leq i,j,r,s\leq k$. Consequently,

\begin{center}
$V^1_{(ijrs)(uvwz)} = \left\{%
\begin{array}{cc}
    1, & \hbox{$\left(\begin{array}{cc} B_i & B_j \\ B_r & B_s \\ B_u & B_v \\ B_w & B_z \end{array}\right) \text{~~is allowed in~~} X$;} \\
    0  & \hbox{\text{~~otherwise~~};} \\
\end{array}%
\right.$
\end{center}

\noindent For computational purposes, let the index set of $V^1$ be ordered using order $O_{\mathcal{H}}$. In particular, the index set of $V_1$ follows the following order:\\

\noindent $(1111),(1112),\ldots,(111k),(1121),\ldots,(112k),\ldots,(11k1),\ldots,(11kk)\\
(1211),(1212),\ldots,(121k),(1221),\ldots,(122k),\ldots,(12k1),\ldots,(12kk)\\
\vdots \\
(1k11),(1k12),\ldots,(1k1k),(1k21),\ldots,(1k2k),\ldots,(1kk1),\ldots,(1kkk)\\
(2111),(2112),\ldots,(211k),(2121),\ldots,(212k),\ldots,(21k1),\ldots,(21kk)\\
(2211),(2212),\ldots,(221k),(2221),\ldots,(222k),\ldots,(22k1),\ldots,(22kk)\\
\vdots \\
(2k11),(2k12),\ldots,(2k1k),(2k21),\ldots,(2k2k),\ldots,(2kk1),\ldots,(2kkk)\\
\vdots \\
(k111),(k112),\ldots,(k11k),(k121),\ldots,(k12k),\ldots,(k1k1),\ldots,(k1kk)\\
(k211),(k212),\ldots,(k21k),(k221),\ldots,(k22k),\ldots,(k2k1),\ldots,(k2kk)\\
\vdots \\
(kk11),(kk12),\ldots,(kk1k),(kk21),\ldots,(kk2k),\ldots,(kkk1),\ldots,(kkkk)$.\\

It may be noted that for any index $\left(\begin{array}{cc} B_i & B_j \\ B_r & B_s \end{array}\right)$ for the matrix $V^1$, the corresponding row can be determined as follows:
\begin{enumerate}
\item If $\left(\begin{array}{cc} B_i & B_j \\ B_r & B_s \end{array}\right)$ is forbidden, then the corresponding row is zero row.\\

\item If $\left(\begin{array}{cc} B_i & B_j \\ B_r & B_s \end{array}\right)$ is allowed then the corresponding row is $R^0_{rs}\bigotimes H^0$.\\
\end{enumerate}

\textbf{Finally let the index set of $V^1$ be ordered using the order $O_{\mathcal{V}}$.}\\

Note that if $\left(\begin{array}{cc} B_i & B_j \\ B_r & B_s \end{array}\right)$ is not allowed then it cannot be vertically aligned with another square to obtain a allowed pattern and hence the corresponding row is the zero row. Further, as the size of the squares $B_k$ is determined by maximum possible length or breadth among patterns from $\mathcal{S}$, any pattern from $\mathcal{S}$ cannot be spread beyond a square of size $2l$. Thus, validity of any given pattern can be verified by examining the validity of all the squares of size $2l$ in the given pattern. Consequently, $\left(\begin{array}{cc} B_i & B_j \\ B_r & B_s \end{array}\right)$  is vertically compatible with $\left(\begin{array}{cc} B_u & B_v \\ B_w & B_z \end{array}\right)$ if and only if $\left(\begin{array}{cc} B_i & B_j \\ B_r & B_s \end{array}\right)$ , $\left(\begin{array}{cc} B_r & B_s \\ B_u & B_v \end{array}\right)$ and $\left(\begin{array}{cc} B_u & B_v \\ B_w & B_z \end{array}\right)$ are allowed. As $(R^0_{rs})_{uv}$ and $(R^0_{uv})_{wz}$ are $1$ if and only if $\left(\begin{array}{cc} B_r & B_s \\ B_u & B_v \end{array}\right)$ and $\left(\begin{array}{cc} B_u & B_v \\ B_w & B_z \end{array}\right)$ are allowed, their product characterizes the validity of the block $\left(\begin{array}{cc} B_i & B_j \\ B_r & B_s \\ B_u & B_v\\ B_w & B_z \end{array}\right)$ under the validity of $\left(\begin{array}{cc} B_i & B_j \\ B_r & B_s \end{array}\right)$ and hence $(R^0_{rs})\bigotimes H^0$ is the row corresponding to $\left(\begin{array}{cc} B_i & B_j \\ B_r & B_s \end{array}\right)$ (under the validity of $\left(\begin{array}{cc} B_i & B_j \\ B_r & B_s \end{array}\right)$). Thus, the constructed matrix indeed characterizes the vertical compatibility of the squares of size $2l$ (and hence characterizes all rectangles of size $4l\times 2l$ allowed for the shift space $X$).\\

Let us investigate the structure of $V_1$ in detail. Note that the index set of $V^1$ (now ordered using the order $O_{\mathcal{V}}$) can be viewed as the block matrix of the form\\

\centerline{{$V^1 = \left(%
\begin{array}{cccc}
  S^0_{11} & S^0_{12} & \ldots & S^0_{1k^2} \\
  S^0_{21} & S^0_{22} & \ldots & S^0_{2k^2} \\
  \vdots & \vdots &        & \vdots  \\
  S^0_{k^21} & S^0_{k^22} & \ldots & S^0_{k^2k^2} \\
\end{array}%
\right)$}}
\vskip 0.5cm
where each $S^0_{ij}$ is a $k^2\times k^2$ matrix. Note that for any $i\in\{1,2,\ldots,k^2\}$, there exists unique pair $(p_i,q_i)$, ($p_i\in\{0,2,\ldots,k-1\},~~ q_i\in\{1,2,\ldots,k\}$) such that $i=kp_i+q_i$. Identifying $i$ with $(p_i+1~~q_i)~~\forall i=1,2,\ldots,k^2$, the matrix $V^1$ can be represented as  \\

\centerline{{$V^1 = \left(%
\begin{array}{ccccccccccccc}
S^0_{(11)(11)} & S^0_{(11)(12)} & \ldots & S^0_{(11)(1k)} & S^0_{(11)(21)} & S^0_{(11)(22)} & \ldots & S^0_{(11)(2k)} & \ldots & S^0_{(11)(k1)} & S^0_{(11)(k2)} & \ldots & S^0_{(11)(kk)}\\
S^0_{(12)(11)} & S^0_{(12)(12)} & \ldots & S^0_{(12)(1k)} & S^0_{(12)(21)} & S^0_{(12)(22)} & \ldots & S^0_{(12)(2k)} & \ldots & S^0_{(12)(k1)} & S^0_{(12)(k2)} & \ldots & S^0_{(12)(kk)}\\
  \vdots & \vdots &    & \vdots & \vdots & \vdots &   & \vdots &  & \vdots & \vdots &   & \vdots \\
S^0_{(1k)(11)} & S^0_{(1k)(12)} & \ldots & S^0_{(1k)(1k)} & S^0_{(1k)(21)} & S^0_{(1k)(22)} & \ldots & S^0_{(1k)(2k)} & \ldots & S^0_{(1k)(k1)} & S^0_{(1k)(k2)} & \ldots & S^0_{(1k)(kk)}\\
S^0_{(21)(11)} & S^0_{(21)(12)} & \ldots & S^0_{(21)(1k)} & S^0_{(21)(21)} & S^0_{(21)(22)} & \ldots & S^0_{(21)(2k)} & \ldots & S^0_{(21)(k1)} & S^0_{(21)(k2)} & \ldots & S^0_{(21)(kk)}\\
S^0_{(22)(11)} & S^0_{(22)(12)} & \ldots & S^0_{(22)(1k)} & S^0_{(22)(21)} & S^0_{(22)(22)} & \ldots & S^0_{(22)(2k)} & \ldots & S^0_{(22)(k1)} & S^0_{(22)(k2)} & \ldots & S^0_{(22)(kk)}\\
  \vdots & \vdots &    & \vdots & \vdots & \vdots &   & \vdots &  & \vdots & \vdots &   & \vdots \\
S^0_{(2k)(11)} & S^0_{(2k)(12)} & \ldots & S^0_{(2k)(1k)} & S^0_{(2k)(21)} & S^0_{(2k)(22)} & \ldots & S^0_{(2k)(2k)} & \ldots & S^0_{(2k)(k1)} & S^0_{(2k)(k2)} & \ldots & S^0_{(2k)(kk)}\\
  \vdots & \vdots &    & \vdots & \vdots & \vdots &   & \vdots &  & \vdots & \vdots &   & \vdots \\
S^0_{(k1)(11)} & S^0_{(k1)(12)} & \ldots & S^0_{(k1)(1k)} & S^0_{(k1)(21)} & S^0_{(k1)(22)} & \ldots & S^0_{(k1)(2k)} & \ldots & S^0_{(k1)(k1)} & S^0_{(k1)(k2)} & \ldots & S^0_{(k1)(kk)}\\
S^0_{(k2)(11)} & S^0_{(k2)(12)} & \ldots & S^0_{(k2)(1k)} & S^0_{(k2)(21)} & S^0_{(k2)(22)} & \ldots & S^0_{(k2)(2k)} & \ldots & S^0_{(k2)(k1)} & S^0_{(k2)(k2)} & \ldots & S^0_{(k2)(kk)}\\
  \vdots & \vdots &    & \vdots & \vdots & \vdots &   & \vdots &  & \vdots & \vdots &   & \vdots \\
S^0_{(kk)(11)} & S^0_{(kk)(12)} & \ldots & S^0_{(kk)(1k)} & S^0_{(kk)(21)} & S^0_{(kk)(22)} & \ldots & S^0_{(kk)(2k)} & \ldots & S^0_{(kk)(k1)} & S^0_{(kk)(k2)} & \ldots & S^0_{(kk)(kk)}\\
\end{array}%
\right)$}}
\vskip 0.5cm

To understand the generated submatrices better, let the $k^4$ rows (columns) be divided into $k^2$ groups of $k^2$ rows (columns) each. Note that if $i=kp_i+q_i$ ($i\in\{1,2,\ldots,k^2\}$), then the rows (columns) of the $i$-th group are indexed by squares of the form $\left(\begin{array}{cc} B_{p_i+1} & * \\ B_{q_i} & * \end{array} \right)$. As $S^0_{ij}$ verifies the vertical compatibility of $i$-th group with the $j$-th group, entries of any $S^0_{(pq)(rs)}$ characterize all rectangles of size $4l\times 2l$ whose left half (which is a rectangle of size $4l\times l$) is $\left(\begin{array}{l} B_p \\ B_q \\ B_r \\ B_s \end{array}\right)$. It is worth mentioning that for $i=kp_i+q_i,~~j=kp_j+q_j$, $S^0_{ij}$ is same as $S^0_{(p_i+1~~ q_i)(p_j+1~~ q_j)}$ and the expression above is another way of representing the same matrix.\\

Let $H^1$ be the matrix indexed by the set of rectangles of size $4l\times 2l$ indicating horizontal compatibility of the indices. As any rectangle of size $4l\times 2l$ is of the form $\left(\begin{array}{cc} B_i & B_u \\ B_j & B_v \\ B_r & B_w \\ B_s & B_z \end{array}\right)$, $H^1$ is a $k^8\times k^8$ matrix equivalently indexed by rectangles of the form  $\left(\begin{array}{cc} B_i & B_u \\ B_j & B_v \\ B_r & B_w \\ B_s & B_z \end{array}\right)$. For notational convenience, let the index $\left(\begin{array}{cc} B_i & B_u \\ B_j & B_v \\ B_r & B_w \\ B_s & B_z \end{array}\right)$ be denoted by $(ijrsuvwz)$.\\

For computational purposes, let the index set of $H^1$ be ordered using order $O_{\mathcal{V}}$. It may be noted that for any index $\left(\begin{array}{cc} B_i & B_u \\ B_j & B_v \\ B_r & B_w \\ B_s & B_z \end{array}\right)$ for the matrix $H^1$, the corresponding row can be determined as follows:
\begin{enumerate}
\item If $\left(\begin{array}{cc} B_i & B_u \\ B_j & B_v \\ B_r & B_w \\ B_s & B_z \end{array}\right)$ is forbidden, then the corresponding row is zero row.\\

\item If $\left(\begin{array}{cc} B_i & B_u \\ B_j & B_v \\ B_r & B_w \\ B_s & B_z \end{array}\right)$ is allowed then the corresponding row is $S^0_{(uv)(wz)}\bigotimes V^1$.\\
\end{enumerate}

\textbf{Finally let the index set of $H^1$ be ordered using the order $O_{\mathcal{H}}$.}\\

Note that if a $4l\times 2l$ block is not allowed then it cannot be horizontally aligned with another $4l\times 2l$ block to obtain a allowed pattern and hence the corresponding row is the zero row. For any given pattern $\mathcal{P}$, as establishing the validity of all squares of size $2l$ (or rectangles or squares of greater size) embedded in $\mathcal{P}$ is sufficient to establish the validity of $\mathcal{P}$ for $X$, $\left(\begin{array}{cc} B_i & B_u \\ B_j & B_v \\ B_r & B_w \\ B_s & B_z\end{array}\right)$ is horizontally compatible with $\left(\begin{array}{cc} B_{i'} & B_{u'} \\ B_{j'} & B_{v'} \\ B_{r'} & B_{w'} \\ B_{s'} & B_{z'}\end{array}\right)$ if and only if $\left(\begin{array}{cc} B_u & B_{i'} \\ B_v & B_{j'} \\ B_w & B_{r'} \\ B_z & B_{s'}\end{array}\right)$ and $\left(\begin{array}{cc} B_{i'} & B_{u'} \\ B_{j'} & B_{v'} \\ B_{r'} & B_{w'} \\ B_{s'} & B_{z'}\end{array}\right)$ are allowed in $X$ (under allowedness of $\left(\begin{array}{cc} B_i & B_u \\ B_j & B_v \\ B_r & B_w \\ B_s & B_z\end{array}\right)$). Finally, as entries of any submatrix $S^0_{(i'j')(r's')}$ characterizes all allowed rectangles of size $4l\times 2l$ with fixed half $\left(\begin{array}{c} B_{i'}  \\ B_{j'} \\ B_{r'} \\ B_{s'}\end{array}\right)$, $S^0_{(uv)(wz)}\bigotimes V^1$ indeed is the row (of $H^1$) corresponding to $\left(\begin{array}{cc} B_i & B_u \\ B_j & B_v \\ B_r & B_w \\ B_s & B_z\end{array}\right)$ and the computation of $H^1$ is complete. \\

\textbf{Step 3: Computation of $V^{n+1}$ and $H^{n+1}$}\\

Let us assume that the matrices $V^0,H^0,V^1,H^1,\ldots,V^n,H^n$ have been computed. Then, the matrix $H^n$ can be visualized as\\

\centerline{{$H^n = \left(%
\begin{array}{cccc}
  R^n_{11} & R^n_{12} & \ldots & R^n_{1k^{4^{n}}} \\
  R^n_{21} & R^n_{22} & \ldots & R^n_{2k^{4^{n}}}\\
  \vdots & \vdots &        & \vdots  \\
  R^n_{k^{4^{n}}1} & R^n_{k^{4^{n}}2} & \ldots & R^n_{k^{4^{n}}k^{4^{n}}} \\
\end{array}%
\right)$}}
\vskip 0.5cm

where each $R^n_{ij}$ is a $k^{4^{n}}\times k^{4^{n}}$ matrix. As $H^n$ establishes the horizontal compatibility of rectangles of size $2^{n+1} l\times 2^n l$, entries of $R^n_{ij}$ establishes the validity of squares of size $2^{n+1} l\times 2^{n+1} l$ (which are indices of $V^{n+1}$) with a fixed top half. Consequently, $R_{ij}\bigotimes H^n$ determines the rows of $V^{n+1}$ (under allowedness of the index), when the index set of $V^{n+1}$ is ordered using order $O_{\mathcal{H}}$. Further, note that the matrix $V^{n+1}$ (ordered using order $O_{\mathcal{V}}$ after computation) can be visualized as: \\

\centerline{{$V^{n+1} = \left(%
\begin{array}{cccc}
  S^{n+1}_{11} & S^{n+1}_{12} & \ldots & S^{n+1}_{1k^{2.4^{n}}} \\
  S^{n+1}_{21} & S^{n+1}_{22} & \ldots & S^{n+1}_{2k^{2.4^{n}}}\\
  \vdots & \vdots &        & \vdots  \\
  S^{n+1}_{k^{2.4^{n}}1} & S^{n+1}_{k^{2.4^{n}}2} & \ldots & S^{n+1}_{k^{2.4^{n}}k^{2.4^{n}}} \\
\end{array}%
\right)$}}
\vskip 0.5cm

where each $S^{n+1}_{ij}$ is a $k^{2.4^{n}}\times k^{2.4^{n}}$ matrix. As $V^{n+1}$ establishes the vertical compatibility of squares of size $2^{n+1} l$, entries of $S^{n+1}_{ij}$ establishes the validity of rectangles of size $2^{n+2} l \times 2^{n+1} l$ (which is an index of $H^n$) with a fixed left half. Consequently, $S^{n+1}_{ij}\bigotimes V^{n+1}$ determines the rows of $H^{n+1}$ (when the index set of $H^{n+1}$ is ordered using order $O_{\mathcal{V}}$) and thus computation of $V^{n+1}$ and $H^{n+1}$ is complete.

Finally, as entries of $V^{n+1}~~(H^{n+1})$ characterize all allowed rectangles (squares) of size $2^{n+2} l\times 2^{n+1} l$ ($2^{n+2} l\times 2^{n+2} l$ respectively), the shift space $X$ is non-empty if and only if each $V^i$ and $H^i$ are non-zero for any $i\in\mathbb{N}$. Further, as any element of $X$ can be viewed as a limit of finite blocks arising from $V^i(\text{or~~} H^i)$, elements of the shift space are precisely the limits of the blocks arising from $V^i(\text{or~~} H^i)$.
\end{proof}

\begin{Remark}
The above result provides an iterative procedure to generate arbitrarily large blocks for a shift space generated by finitely many forbidden blocks of finite size. Note that although the algorithm generates rectangles (squares) of size $2^{n+1} l\times 2^{n} l$ ($2^{n+1} l\times 2^{n+1} l$ respectively) using all squares (rectangles) of size $2^{n} l\times 2^{n} l$ ($2^{n+1} l\times 2^{n} l$ respectively), the matrices $V^i~~(H^i)$ can be indexed by a smaller collection of all allowed squares (rectangles) to generate the same set and hence the order of the matrices $V^i~~(\text{and}~~ H^i)$ can be reduced. Further, as any compatible vertical (horizontal) alignment of indices of $V^i~~(H^i)$ is an index for $H^i~~(V^{i+1})$, the number of $1$'s in $V^i~~(H^i)$ characterizes the order of $H^i~~(V^{i+1})$ (while working with matrices $V^i~~(H^i)$ of reduced size). Note that while the matrices the matrices $V^i$ (and $H^i$) are computed with the indices being ordered using $O_{\mathcal{H}}$ (and $O_{\mathcal{V}}$ respectively), the indices are re-ordered after computation using $O_{\mathcal{V}}$ (and $O_{\mathcal{H}}$ respectively). Such an arrangement is useful as such an ordering of the index set helps represent the generated matrix in the form of a block matrix where each block characterizes rectangles (squares) of larger size with a fixed left (top) half which is useful to compute the rows of the next matrix $H^n~~(V^{n+1})$ (via the product $S^n_{ij}\bigotimes V^n \text{~~and~~} (R^n_{ij}\bigotimes H^n)$ respectively). Note that while the ability to construct arbitrarily large blocks for the space $X$ guarantees non-emptiness of the shift space $X$, any element of the shift space is a limit of arbitrarily large blocks allowed for $X$. Consequently, the shift space $X$ is non-empty if and only if the matrices $V^i$ and $H^i$ are non-zero at each iteration. Further, note that  for a $d$-dimensional shift space, a similar algorithm yields a sequence of matrices $M^1_1, M^2_1,\ldots M^d_1,M^1_2, M^2_2,\ldots M^d_2\ldots,M^1_n, M^2_n,\ldots M^d_n\ldots$ (where $M^i_n$ characterises the possible extensions in the $i$-th direction at $n$-th iteration) such that the $d$-dimensional shift space is non-empty if and only if $M^i_n~~ (i=1,2,\ldots,d)$ are non-zero at each iteration. Thus we get the following corollary.\\
\end{Remark}

\begin{Cor}
For every $d$-dimensional shift of finite type, there exists a sequence of finite matrices characterizing the elements of $X$.
\end{Cor}

\begin{proof}
The proof is a natural extension of the $2$ dimensional case and follows directly by first extending the allowed $d$-dimensional cubes in each of the $d$ directions and iteratively generating cubes (cuboids) of arbitrarily large size. For the sake of clarity, we provide an outline of the proof below.\\

{\bf Computing Generating Matrices}\\

Let $X$ be a $d$-dimensional shift space generated by a finite set of forbidden patterns. By Lemma \ref{sb}, there exists a set of forbidden cubes $\mathcal{S}$ (all of same size, say $l$) generating the shift space $X$. Let $e_1,e_2,\ldots,e_d$ be the set of $d$ mutually orthogonal directions (along the directions of standard basis vectors of $\mathbb{R}^d$) and let $B_1,B_2,\ldots,B_k$ be the the collection of all allowed cubes of size $l$. Let $M^1_1$ be the $k\times k$ matrix (indexed by $B_1,B2,\ldots,B_k$)  indicating compatibility of the cubes $B_i$ in $e_1$ direction. Consequently, entries of $M^1_1$ characterizes the validity of all blocks of size $2l\times \underbrace{l\times l\times\ldots\times l}_{d-1~~ times}$  for the space $X$. Similarly, if $M^2_1$ be the matrix indexed by all allowed blocks of size $2l\times \underbrace{l\times l\times \ldots\times l}_{d-1~~ times}$ indicating the compatibility of the indices in $e_2$ direction then, entries of $M^2_1$ characterizes the validity of all $2l\times 2l\times \underbrace{l\times l\times\ldots\times l}_{d-2~~ times}$  for the space $X$. Inductively, for $i=3,4,\ldots,d$, if $M^i_1$ is indexed by all allowed blocks of size $\underbrace{2l\times 2l\times\ldots\times 2l}_{i-1~~ times}\times \underbrace{l\times l\times\ldots\times l}_{d-i+1~~ times}$ indicating the compatibility of the indices in $i$-th direction, entries of $M^i_1$ characterize the validity of all blocks of size $\underbrace{2l\times 2l\times\ldots\times 2l}_{i~~ times}\times \underbrace{l\times l\times\ldots\times l}_{d-i~~ times}$ . In particular, $M^d_1$  is a matrix indexed by all allowed blocks of size $\underbrace{2l\times\ldots 2l}_{d-1~~ times}\times l$ indicating the compatibility of the indices in the direction $e^d$ and hence characterizes all allowed cubes of size $2l$. \\

{\bf Application of Induction and Generating Cubes (Cuboids) of Arbitrarily Large Size}\\

Let us assume that $M^1_1, M^2_1,\ldots M^d_1,M^1_2, M^2_2,\ldots M^d_2\ldots,M^1_n, M^2_n,\ldots M^d_n$ have been computed. Note that as $M^i_n$ is indexed by all allowed blocks of size $\underbrace{2^n l\times 2^nl\times\ldots\times 2^nl}_{i-1~~ times}\times \underbrace{2^{n-1}l\times 2^{n-1}l\times\ldots\times 2^{n-1}l}_{d-i+1~~ times}$, entries of $M^i_n$ characterize all allowed blocks of size $\underbrace{2^n l\times 2^nl\times\ldots\times 2^nl}_{i~~ times} \times \underbrace{2^{n-1}l\times 2^{n-1}l\times\ldots\times 2^{n-1}l}_{d-i~~ times}$.

Further, in order for the proof of Theorem \ref{2d} to be extendable to the $d$-dimensional case, while the indices of $M^i_n$ need to be ordered using $(i-1)$-th direction (mod $d$) while computing $M^i_n$, they need to be ordered using $(i+1)$-th direction (mod $d$) after computation, to facilitate the computation of the next matrix. Finally, note that as the sequence $M^i_n$ generates arbitrarily large cubes (and cuboids), the process yields an element of the shift if each $M^i_n$ is non-empty. Further, as any element of sequence can be viewed as a limit of arbitrarily large cubes (or cuboids), the elements of the shift space are characterized by the sequence generated.
\end{proof}

\begin{Remark}
The above corollary extends the algorithm given in Theorem \ref{2d} for generating elements of $d$-dimensional shift space. The algorithm generates cubes (cuboids) of arbitrarily large size by iteratively extending the allowed blocks in each of the $d$ directions. Note that as compatibility of two blocks in one of the directions cannot indicate its extension in any other direction, the matrices $M^1_1,M^2_1,\ldots,M^d_1$ are independent of each other (cannot be determined from each other). Further, as the size of the matrices $B_1,B_2,\ldots,B_k$ is the maximum possible length (in any of the directions) of the generating set of forbidden patterns $\mathcal{P}$, to examine the validity of any given pattern in $\mathbb{R}^d$, it is sufficient to examine the validity of all cubes of size $2l$ (or cubes/cuboids of larger size). Finally, as any element of $X$ can be visualized as a limit of finite cubes (cuboids) and generation of arbitrarily large cubes (cuboids) yields an element of $X$, the sequence generated characterizes the elements of $X$. To realize the convergence mathematically, one may visualize the allowed cube as an element which is obtained by placing the cube (center of the cube) at the origin and assigning $0$ at all the other places. Consequently, a sequence of arbitrarily large blocks yields a cauchy sequence in the full shift without forbidden blocks in the central cube (which is of arbitrarily large size) and hence converges to an element of $X$.
\end{Remark}

\bibliography{xbib}

\end{document}